\documentclass[11pt,epsf]{article}
\usepackage{graphicx}
\usepackage{amsthm}
\usepackage{amsfonts}
\usepackage{amssymb}
\title{Two properties of the partial theta function}
\author{Vladimir Petrov Kostov\\ 
Universit\'e C\^ote d'Azur, CNRS, LJAD, France,\\  
e-mail: vladimir.kostov@unice.fr} 
\date{}
\newtheorem{tm}{Theorem}

\newtheorem{rem}[tm]{Remark}
\newtheorem{rems}[tm]{Remarks}
\newtheorem{lm}[tm]{Lemma}

\newtheorem{nota}[tm]{Notation}
\begin{document} 
\maketitle 
\begin{abstract}
  For the partial theta function $\theta (q,z):=\sum _{j=0}^{\infty}q^{j(j+1)/2}z^j$,
  $q$, $z\in \mathbb{C}$, $|q|<1$, 
  we prove that its zero set is connected. This set is smooth at every point
  $(q^{\flat},z^{\flat})$ such that $z^{\flat}$ is a simple or double zero of
    $\theta (q^{\flat},.)$. For $q\in (0,1)$, $q\rightarrow 1^-$ and $a\geq e^{\pi}$, there are $o(1/(1-q))$ and $(\ln (a/e^{\pi}))/(1-q)+o(1/(1-q))$ real zeros of $\theta (q,.)$ in the intervals $[-e^{\pi},0)$ and $[-a,-e^{-\pi}]$ respectively (and none in $[0,\infty)$). For $q\in (-1,0)$, $q\rightarrow -1^+$ and $a\geq e^{\pi /2}$, there are $o(1/(1+q))$ real zeros of $\theta (q,.)$ in the interval $[-e^{\pi /2},e^{\pi /2}]$ and $(\ln (a/e^{\pi /2})/2)/(1+q)+o(1/(1+q))$ in each of the intervals $[-a,-e^{\pi /2}]$ and $[e^{\pi /2},a]$.


\end{abstract}

\section{Introduction}
\subsection{Definition of the partial theta function}

For $q\in \mathbb{D}_1$, $z\in \mathbb{C}$, where $\mathbb{D}_r$ stands for
the open disk of radius $r$ centered at $0\in \mathbb{C}$, one defines the
{\em partial theta function} by the formula

\begin{equation}\label{eqdefi}
  \theta (q,z):=\sum _{j=0}^{\infty}q^{j(j+1)/2}z^j~.
\end{equation}
This terminology is explained
by the resemblance of formula (\ref{eqdefi}) with the one defining the
{\em Jacobi theta function} $\Theta (q,z):=\sum _{j=-\infty}^{\infty}q^{j^2}z^j$;
the word ``partial'' refers to the summation in the case of $\theta$ taking
place only over the nonnegative values of $j$. One has
$\theta (q^2,z/q)=\sum _{j=0}^{\infty}q^{j^2}z^j$. We consider $q$ as a parameter
and $z$ as a variable. For each $q$ fixed, $\theta (q,.)$ is an entire
function.

The function $\theta$ has been studied as Ramanujan-type series in~\cite{Wa}.
Its applications in statistical physics and combinatorics are explained
in~\cite{So}. Other fields, where $\theta$ is used, are the theory of
(mock) modular forms (see~\cite{BrFoRh}) and asymptotic analysis
(see~\cite{BeKi}). Asymptotics, modularity and other properties of partial
and false theta functions are considered in \cite{CMW} with regard to conformal
field theory and representation theory, and in \cite{BFM} when asymptotic
expansions of regularized characters and quantum dimensions of the
$(1,p)$-singlet algebra modules are studied.

A recent domain of interest for $\theta$ (in the case when the parameter
$q$ is real) is the theory of
{\em section-hyperbolic polynomials}, i.e. real univariate polynomials of
degree $\geq 2$ with all roots real negative and such that, when
their highest-degree monomial is deleted, this gives again a polynomial
having only real negative roots. The classical results of Hardy, Petrovitch
and Hutchinson in this direction (see~\cite{Ha}, \cite{Pe} and \cite{Hu})
have been continued in \cite{Ost}, \cite{KaLoVi} and~\cite{KoSh}. Various
analytic properties of $\theta$ are studied in \cite{Ko1}, \cite{Ko2},
\cite{Ko3}, \cite{Ko4}, \cite{Ko5}, \cite{Ko6}, \cite{Ko7}, \cite{Ko8},
\cite{Ko9}, \cite{Ko10}, \cite{Ko11}, \cite{Ko12} and \cite{Ko13}.
See more about $\theta$ in~\cite{AnBe}.

\subsection{The zero set and the spectrum of the partial theta function}

In the present paper we consider the zero set of $\theta$, i.e. the set
$S:=\{ (q,z)\in \mathbb{D}_1\times \mathbb{C}$, $\theta (q,z)=0\}$. In
Section~\ref{secpr} we prove the following theorem:

\begin{tm}\label{maintm}
  The set $S$ is a connected analytic subset of
  $\mathbb{D}_1\times \mathbb{C}$. It is smooth at every point
  $(q^{\flat},z^{\flat})$ such that $z^{\flat}$ is a simple or double zero of
  $\theta (q^{\flat},.)$.
\end{tm}

\begin{rems}\label{rems1}
  {\rm (1) B.~Z.~Shapiro has introduced the notion of {\em spectrum} of
    $\theta$ as the set of values of $q$ for which $\theta (q,.)$ has a
    multiple zero,
    see~\cite{KoSh}. Suppose that $q$ is real, i.e. $q\in (-1,0)\cup (0,1)$ (the case $q=0$ is of little interest since $\theta (0,z)\equiv 1$).
    If $q\in (0,1)$, then $\theta (q,.)$ has infinitely-many real zeros
    and they are all negative. 
    There are also infinitely-many spectral numbers
    $0<\tilde{q}_1<\tilde{q}_2<$ $\cdots$ $<\tilde{q}_k<\cdots <1$,
    $\lim _{k\rightarrow \infty}\tilde{q}_k=1^-$, 
    see~\cite{Ko2}.

    (2) For $q\in (0,\tilde{q}_1)$ (where $\tilde{q}_1=0.3092\ldots$), all
    zeros of $\theta (q,.)$ are real, negative and distinct:
    $\cdots <\xi _2<\xi _1<0$; one has $\theta (q,x)>0$ for $x\in (\xi _{2j+1},\xi _{2j})$ and $\theta (q,x)<0$ for $x\in (\xi _{2j},\xi _{2j-1})$. For
    $q\in (\tilde{q}_k,\tilde{q}_{k+1})\subset (0,1)$, $k\in \mathbb{N}$,
    $\tilde{q}_0:=0$,
    the function
    $\theta (q,.)$ has exactly $k$ pairs of complex conjugate zeros (counted
    with multiplicity). When $q\in (0,1)$ increases and passes through the
    spectral value $\tilde{q}_k$, the two zeros $\xi _{2k-1}$ and $\xi _{2k}$
    coalesce and then form a complex conjugate pair, see~\cite{Ko2}. The index
    $j$ of the zero $\xi _j$ is meaningful as long as $\xi _j$ is real, i.e.
    for $q\in (0,\tilde{q}_{[(j+1)/2]}]$, where $[.]$ stands
  for ``the integer part of''.

    (3) Asymptotic expansions of the
    numbers $\tilde{q}_k$ are
    proved in \cite{Ko3} and \cite{Ko13}. The formula of \cite{Ko13} reads:

    \begin{equation}\label{eqasympt1}
      \tilde{q}_k=1-\pi /2k+(\ln k)/8k^2+O(1/k^2)~~~\, ,~~~\,
      \tilde{y}_k=-e^{\pi }e^{-(\ln k)/4k+O(1/k)}
    \end{equation}
    where $e^{\pi }=23.1407\ldots$ and $\tilde{y}_k<0$ is the double zero of $\theta (\tilde{q}_k,.)$. It is
    the rightmost of its real zeros and $\theta (\tilde{q}_k,.)$
    has a minimum at~$\tilde{y}_k$.

    (4) For $k\in \mathbb{N}^*$, one has $\theta (q,-q^{-k})\in (0,q^k)$,
    see
    Proposition~9 in~\cite{Ko2}. For $q>0$
    small enough, one has sgn$(\theta (q,-q^{-k-1/2}))=(-1)^k$ and
    $|\theta (q,-q^{-k-1/2})|>1$, see Proposition~12 in~\cite{Ko2}.}
    \end{rems}

    \begin{rems}\label{rems1bis}
    {\rm (1) If $q\in (-1,0)$ is sufficiently small, then $\theta (q,.)$ has
      infinitely-many real negative and infinitely-many real positive zeros:
      $\cdots <\xi _4<\xi _2<0<\xi _1<\xi _3<\cdots$ ; one has $\theta (q,x)<0$ for $x\in (\xi _{4j+4},\xi _{4j+2})$ and $x\in (\xi _{4j+1},\xi _{4j+3})$,  and $\theta (q,x)>0$ for $x\in (\xi _{4j+2},\xi _{4j})$, $x\in (\xi _{4j+3},\xi _{4j+5})$ and $x\in (\xi _2,\xi _1)$. For $q\in (-1,0)$, there are also
      infinitely-many spectral numbers, see~\cite{Ko7}. 
    We denote them by
    $\bar{q}_k$, where $-1<\bar{q}_k<0$.  
    
    (2) For $s\geq 1$, one has $-1<\bar{q}_{2s+1}<\bar{q}_{2s-1}<0$, see
    Lemma~4.11 in \cite{Ko7}. For $k$ sufficiently
    large, one has $\bar{q}_{k+1}<\bar{q}_k$, see Lemmas~4.10, 4.11 and 4.17
    in~\cite{Ko7}. The inequality $\bar{q}_{k+1}<\bar{q}_k$ being proved only
    for $k$ sufficiently large we admit the possibility finitely-many
    equalities of the form $\bar{q}_i=\bar{q}_j$ to hold true, where at least
    one of the numbers $i$ and $j$ is even. 
    
(3) When $q\in (-1,0)$
    decreases and passes through a spectral value $\bar{q}_k$, then for
    $k=2s-1$ (resp. $k=2s$), $s\in \mathbb{N}^*$, the zeros
$\xi _{4s-2}$ and $\xi _{4s}$ (resp. $\xi _{4s-1}$ and $\xi _{4s+1}$) 
coalesce. Thus for 
    $q\in (\bar{q}_{k+1},\bar{q}_k)\subset (-1,0)$ and $k$ sufficiently large,
    the function
    $\theta (q,.)$ has exactly $k$ pairs of complex conjugate zeros (counted
    with multiplicity). The zero $\xi _1$ remains real and simple for any
    $q\in (-1,0)$. 

    (4) Asymptotic expansions of the
    numbers $\bar{q}_k$ are
    proved in \cite{Ko7}:

    \begin{equation}\label{eqasympt2}
      \bar{q}_k=1-(\pi /8k)+o(1/k)~~~\, ,~~~\, |\bar{y}_k|=e^{\pi /2}+o(1)~,
      \end{equation}
    where $\bar{y}_k$ is the double zero of $\theta (\bar{q}_k,.)$ and 
    $e^{\pi /2}=4.81477382\ldots$. For $k$ odd (respectively, $k$ even)
    $\theta (\bar{q}_k,.)$ has a local minimum (respectively, maximum) at
    $\bar{y}_k$, and $\bar{y}_k$ is the rightmost of the real negative zeros
    of $\theta (\bar{q}_k,.)$ (respectively, for $k$ sufficiently
    large, $\bar{y}_k$ is the second from the left of the real positive
    zeros of $\theta (\bar{q}_k,.)$).}
    \end{rems}

    \begin{rems}\label{rems1ter}
      {\rm (1) All coefficients in the series (\ref{eqdefi}) are real. Hence a
        priori
    spectral numbers are either real or they form complex conjugate pairs.
    It is proved in \cite{Ko8} that there exists at least one such pair which
    equals $0.4353184958\ldots \pm i0.1230440086\ldots$. Numerical results
    suggest that one should expect there to be infinitely-many such pairs.

    (2) In any set of the form $\mathbb{D}_r\setminus \{ 0\}$, $r\in (0,1)$,
    the number of spectral values of $\theta$ is finite (because the spectrum
    is locally a codimension $1$ analytic subset in
    $\mathbb{D}_1\setminus \{ 0\}$). For any spectral
    number $q$, the function $\theta (q,.)$ has finitely-many multiple zeros,
    see~\cite{Ko5}. The number $\tilde{q}_1=0.3092\ldots$ is the only spectral number of $\theta$ in the disk $\overline{\mathbb{D}_{0.31}}$.
    
    (3) For all spectral numbers $\tilde{q}_j\in (0,1)$ and, for $k$
    sufficiently large, for all spectral numbers $\bar{q}_k\in (-1,0)$,
    it is true that the function
    $\theta (\tilde{q}_j,.)$, resp. $\theta (\bar{q}_k,.)$, has exactly one
    double zero all its other zeros being
    simple (see \cite{Ko2} and \cite{Ko7}). It would be interesting to
    (dis)prove that this is the case of
    any spectral number. If true, this would mean in particular (see
    Theorem~\ref{maintm})
    that $S$ is globally smooth and connected. If false, it would be of
    interest to describe the eventual singularities of~$S$.}
\end{rems}

    \subsection{The limit distribution of the real zeros}

    In the present subsection we consider the case $q\in \mathbb{R}$, i.e.
    $q\in (-1,0)\cup (0,1)$.

    \begin{nota}\label{notalnp}
      {\rm For $q\in (-1,0)\cup (0,1)$, given a finite interval
        $J\subset \mathbb{R}$, we denote by $Z_J(q)$ the number of zeros of
        $\theta (q,.)$ (counted with multiplicity) belonging to $J$. For
        $q\in (0,1)$ and $a\geq e^{\pi}$, we set $\ell _a(q):=Z_{[-a,-e^{\pi}]}(q)$.
        For $q\in (-1,0)$ and $a\geq e^{\pi /2}$, we set
        $n_a(q):=Z_{[-a,-e^{\pi /2}]}(q)$ and $p_a(q):=Z_{[e^{\pi /2},a]}(q)$.}
      \end{nota}

    \begin{tm}\label{tm01}
      (1) For $q\in (0,1)$, one has $Z_{[-e^{\pi},0)}(q)=o(1/(1-q))$. 

        (2) The set of zeros of $\theta (\tilde{q}_k,.)$ (over all
        $k\in \mathbb{N}^*$) is everywhere dense in $(-\infty , -e^{\pi}]$.
      One has $\lim _{q\rightarrow 1^-}\ell _a(q)(1-q)=\ln (a/e^{\pi})$.

      (3) For $q\in (-1,0)$, one has $Z_{[-e^{\pi /2},0)}(q)=o(1/(1+q))$.  

        (4) The set of zeros of $\theta (\bar{q}_{2s-1},.)$ (over all
        $s\in \mathbb{N}^*$) is everywhere dense in $(-\infty , -e^{\pi /2}]$.
      One has
      $\lim _{q\rightarrow 1^-}n_a(q)(1+q)=\ln (a/e^{\pi /2})/2$.

      (5) For $q\in (-1,0)$, one has $Z_{(0,e^{\pi /2}]}(q)=o(1/(1+q))$. 

        (6) The set of zeros of $\theta (\bar{q}_{2s},.)$ (over all
        $s\in \mathbb{N}^*$) is everywhere dense in $[e^{\pi /2}, \infty )$.
      One has
      $\lim _{q\rightarrow 1^-}p_a(q)(1+q)=\ln (a/e^{\pi /2})/2$.
      \end{tm}

    The theorem is proved in Section~\ref{sectm01}.

    \begin{rem}
      {\rm The quantity
        $1/a=\lim _{\varepsilon \rightarrow 0^+}((\ln ((a+\varepsilon )/e^{\pi})-
        \ln (a/e^{\pi })/\varepsilon )$ can be interpreted as limit density
        of the real zeros of $\theta (q,.)$ as $q\rightarrow 1^-$ at the
        point $-a\leq -e^{\pi}$. Similarly for the quantity $1/(2a)$ at
        $\pm a$, $a\geq e^{\pi /2}$, as $q\rightarrow -1^+$. For the rest of the
        real line the limit density is~0. Indeed, for $q\in (0,1)$, there are
        no nonnegative zeros of $\theta (q,.)$; for $0<a<e^{\pi}$, see part~(1)
        of Theorem~\ref{tm01}. For $q\in (-1,0)$, see parts (3) and (5) of
        Theorem~\ref{tm01}.}
      \end{rem}

\section{Proof of Theorem~\protect\ref{maintm}\protect\label{secpr}}

\subsection{Smoothness}

We prove the smoothness first. If $z^{\flat}$ is a simple zero of
$\theta (q^{\flat},.)$, then
$(\partial \theta /\partial z)(q^{\flat}, z^{\flat})\neq 0$ hence

\begin{equation}\label{eqgradnot0}
{\rm Grad}(\theta ) (q^{\flat},z^{\flat})\neq 0
\end{equation}
and $S$ is smooth at $(q^{\flat}, z^{\flat})$. The function $\theta$ satisfies the
following differential equation (see~(\ref{eqdefi})):

\begin{equation}\label{eqdiff}
  2q(\partial \theta /\partial q)=
  z(\partial ^2/\partial z^2)(z\theta )~.
\end{equation}
The right-hand side equals
$2z(\partial \theta /\partial z)+z^2(\partial ^2\theta /\partial z^2)$. If
$z^{\flat}$ is a double zero of $\theta (q^{\flat},.)$, then

$$\theta (q^{\flat},z^{\flat})=(\partial \theta /\partial z)(q^{\flat},z^{\flat})=0
\neq (\partial ^2\theta /\partial z^2)(q^{\flat},z^{\flat})~.$$
One has neither $q^{\flat}=0$, because $\theta (0,.)\equiv 1\neq 0$, nor
$z^{\flat}=0$, because $\theta (q,0)\equiv 1$.
Hence

$$(\partial \theta /\partial q)(q^{\flat},z^{\flat})=
((z^{\flat})^2/2q^{\flat})(\partial ^2\theta /\partial z^2)
(q^{\flat},z^{\flat})\neq 0~.$$
Therefore one
has (\ref{eqgradnot0}), so $S$ is smooth at $(q^{\flat},z^{\flat})$.

\subsection{Separation in modulus\protect\label{subsecsepar}}

For fixed $q\in \mathbb{D}_1\setminus \{ 0 \}$, we denote by $\mathcal{C}_k$,
$k\in \mathbb{N}^*$, the circumference in the $z$-space $|z|=|q|^{-k-1/2}$. When
$q$ is close to $0$, one can enumerate the zeros of $\theta$, because there
exists exactly one zero such that $\xi _k\sim -q^{-k}$ (see Proposition~10 in
\cite{Ko2}). For $0<|q|\leq c_0:=0.2078750206\ldots$, one has

\begin{equation}\label{eqsepar}
  |q|^{-k+1/2}<|\xi _k|<|q|^{-k-1/2}~,
\end{equation}
see Lemma~1 in \cite{Ko8}. In this sense we say that for
$q\in \mathbb{D}_{c_0}\setminus \{ 0\}$, the zeros of $\theta$ are separated in
modulus (that is, their moduli are separated by the circumferences
$\mathcal{C}_k$). We say that, for given $q$, {\em strong separation} of the
zeros of $\theta$ takes place for $k\geq k_0$, if for any $k\geq k_0$, there
exists exactly one zero $\xi _k$ of $\theta$ satisfying conditions
(\ref{eqsepar}).

Set $\alpha _0:=\sqrt{3}/2\pi =0.2756644477\ldots$. The following result can
be found in~\cite{Ko8}:

\begin{tm}\label{tmsepar}
  For $n\geq 5$ and for $|q|\leq 1-1/(\alpha _0n)$, strong separation of the
  zeros of $\theta$ takes place for $k\geq n$.
\end{tm}

Theorem~\ref{tmsepar} has several important corollaries:
\vspace{1mm}

{\em i)} For each path $\gamma \subset \mathbb{D}_1\setminus \{ 0\}$
in the $q$-space which avoids the spectral numbers of $\theta$,
one can define by continuity the zeros of $\theta$ as functions of $q$ as
$q$ varies along $\gamma$. One
can find $k\in \mathbb{N}$ such that
$\gamma \subset \mathbb{D}_{1-1/(\alpha _0k)}$. For $n\geq k$, the zero $\xi _n$
is an analytic function in $q\in \mathbb{D}_{1-1/(\alpha _0k)}$. Thus the
indices of the zeros $\xi _n$ are meaningful for $n\geq k$ and
$q\in \mathbb{D}_{1-1/(\alpha _0k)}$.
\vspace{1mm}

{\em ii)} Denote by $\Gamma$ the spectrum of $\theta$. If
$\gamma \subset D:=\mathbb{D}_{1-1/(\alpha _0k)}\setminus
\{ \Gamma \cup \{ 0\} \}$
is a loop, then the zeros of $\theta$ lying inside
$\mathcal{C}_k$ might undergo a {\em monodromy} as $q$ varies along $\gamma$,
i.e. a permutation which depends on
the class of homotopy equivalence of $\gamma$ in
$D$. Therefore it might not be possible to correctly
define the indices of these zeros for $q\in D$. 
\vspace{1mm}

{\em iii)} For no 
$q_*\in \mathbb{D}_1\setminus \{ 0\}$ does a zero of $\theta$ go to infinity as
$q\rightarrow q_*$.
That is, zeros are not born and do not disappear at infinity.
\vspace{1mm}

{\em iv)} For $(0,1)\ni q=\tilde{q}_j\in \Gamma$, the function $\theta (q,.)$
has one double zero and infinitely-many simple zeros,
see part (3) of Remarks~\ref{rems1ter} and part~(3) of Remarks~\ref{rems1}. The double zero is a Morse critical
point for $\theta$. Suppose that $\gamma$ is a small loop in
$\mathbb{D}_1\setminus \{ 0\}$
circumventing $\tilde{q}_j$. Then the two zeros $\xi _{2j-1}$ and $\xi _{2j}$ of
$\theta (q,.)$ which coalesce for
$q=\tilde{q}_j$
are exchanged as $q$ varies along $\gamma$. For
$(-1,0)\ni q=\bar{q}_k\in \Gamma$, $k=2s-1$ or $2s$, where $s\geq 1$ is
sufficiently large, the same remark
applies to the zeros $\xi _{4s-2}$ and $\xi _{4s}$ or $\xi _{4s-1}$ and
$\xi _{4s+1}$,
see part (3) of Remarks~\ref{rems1bis}. For the remaining values of $k$, if,
say,  $p$  spectral values $\bar{q}_i$ coincide, then the function
$\theta (\bar{q}_i,.)$ has $p$ double real zeros (its other real zeros are
simple) and the monodromy defined by the 
class of homotopy equivalence of $\gamma$ exchanges 
the zeros in $p$ non-intersecting couples of zeros (which are close to the
double zeros of $\theta (\bar{q}_i,.)$). 
\vspace{1mm}

{\em v)} Theorem~\ref{tmsepar} implies that the monodromy around
$0\in \mathbb{D}_1$ is trivial.

\subsection{Connectedness of $S$}

When $q\in \mathbb{D}_{c_0}\setminus \{ 0\}$, the zeros $\xi _j$
can be considered as analytic functions in $q$. We discuss the
possible monodromies which they can undergo when the parameter $q$ runs
along certain loops in $\mathbb{D}_1\setminus \{ 0\}$. First of all we recall
that for $q\in (0,\tilde{q}_j)$, the zeros
$0>\xi _{2j-1}>\xi _{2j}>\xi _{2j+1}>\cdots$ are simple, real negative and
continuously depending on $q$, see part~(2) of Remarks~\ref{rems1}; for
$q=\tilde{q}_j$, the zeros $\xi _{2j-1}$ and $\xi _{2j}$ coalesce.

Suppose that $a\in (0,c_0)$ and that
$\mathcal{C}^{\sharp}\subset \mathbb{D}_1\setminus \{ 0\}$ is a small
circumference of radius $\varepsilon$
centered at the spectral number $\tilde{q}_j$, see parts (1) and (2) of
Remarks~\ref{rems1}; no spectral number other than $\tilde{q}_j$ belongs to
the circumference $\mathcal{C}^{\sharp}$ or to its interior. Define
$\gamma _j\subset \mathbb{D}_1\setminus \{ 0\}$ as the path
consisting of the segment
$\sigma _+:=[a, \tilde{q}_j-\varepsilon ]\subset \mathbb{R}$,
the circumference $\mathcal{C}^{\sharp}$
(which is run, say, counterclockwise) and the segment
$\sigma _-:=[\tilde{q}_j-\varepsilon , a]$. Hence if one considers the analytic
continuation of the function $\xi _{2j-1}$ (resp. $\xi _{2j}$) along the loop
$\gamma _j$, the result will be the function $\xi _{2j}$ (resp. $\xi _{2j-1}$),
see {\em iv)} in Subsection~\ref{subsecsepar}. We denote this symbolically by
$\gamma _j:\xi _{2j-1}\leftrightarrow \xi _{2j}$. If we need to indicate only the
image of $\xi _{2j-1}$ we might write $\gamma _j:\xi _{2j-1}\mapsto \xi _{2j}$.

\begin{rem}\label{remmodify1}
  {\rm For $j>1$, the two segments $\sigma _{\pm}$
    of the path $\gamma _j$ pass through the
    spectral numbers $\tilde{q}_1$, $\ldots$, $\tilde{q}_{j-1}$.
    If one insists the path $\gamma _j$ to bypass all spectral numbers
    $\tilde{q}_j$, then one should modify $\gamma _j$. Namely, parts of the 
    two segments $\sigma _{\pm}$ which are segments of the form
    $\sigma _s:=[\tilde{q}_s-\varepsilon ',\tilde{q}_s+\varepsilon ']$,
    $0<\varepsilon '\ll \varepsilon$, $1\leq s\leq j-1$,
    should be replaced by small half-circumferences with diameters $\sigma _s$
    which bypass the
    spectral numbers $\tilde{q}_s$ from above or below.}
  \end{rem}

Suppose that $q\in (-1,0)$. We will make use of Remarks~\ref{rems1bis}. 
We construct a path $\delta _j$ consisting of a segment
$\tau _-:=[-a, \bar{q}_j+\varepsilon ]\subset \mathbb{R}$, $-c_0<-a<0$ ($c_0$ is defined at the beginning of Subsection~\ref{subsecsepar}), a 
circumference $\mathcal{C}^{\triangle}\subset \mathbb{D}_1\setminus \{ 0\}$ of
radius $\varepsilon$ centered at
$\bar{q}_j$ (and run, say, counterclockwise) and the segment
$\tau _+:=[\bar{q}_j+\varepsilon ,-a]$. If $\bar{q}_{j_1}\neq \bar{q}_j$,
then the spectral number
$\bar{q}_{j_1}$ does not belong to $\mathcal{C}^{\triangle}$ or to its interior.

Suppose that $j=2s-1$ (resp. $j=2s$). If one considers the analytic
continuation of the functions $\xi _{4s-2}$ and $\xi _{4s}$ (resp. of $\xi _{4s-1}$
and $\xi _{4s+1}$) along the loop
$\delta _j$, the result will be that the functions $\xi _{4s-2}$ and $\xi _{4s}$
(resp. $\xi _{4s-1}$ and $\xi _{4s+1}$) exchange their values,
see {\em iv)} in Subsection~\ref{subsecsepar}. We denote this symbolically by
$\delta _{2s-1}:\xi _{4s-2}\leftrightarrow \xi _{4s}$ or
$\delta _{2s}:\xi _{4s-1}\leftrightarrow \xi _{4s+1}$.

\begin{rem}\label{remmodify2}
  {\rm Similarly to what was done with the path $\gamma _j$, see
    Remark~\ref{remmodify1}, one can modify the path $\delta _j$ so that it
    should pass through no spectral value of $\theta$. We do not claim,
    however, that an equality of the form $\bar{q}_{j_1}=\bar{q}_{j_2}$,
    $j_1\neq j_2$, 
    does not take place (this is not proved in \cite{Ko7}; see part (2) of
    Remarks~\ref{rems1bis}). Nevertheless, even
    if such an equality holds true, then it does not affect our reasoning,
    because when $q$
    runs along $\mathcal{C}^{\triangle}$ close to the coinciding spectral numbers
    $\bar{q}_{j_1}$ and $\bar{q}_{j_2}$, the exchange of zeros $\xi _i$ which
    occurs concerns two couples of zeros with no zero in common.}
\end{rem}

By combining the monodromies defined by the paths $\gamma _j$ and $\delta _j$
one can obtain any monodromy $\xi _k\mapsto \xi _m$. Indeed, denote
by $\eta _+$ a half-circumference centered at $0$, of radius $a$, 
belonging to the upper half-plane (hence the segment $[-a,a]$ is
its diameter) and run counterclockwise, by $\eta _-$ the same
half-circumference run clockwise, by
$\gamma _j\gamma _{\ell}$ the concatenation of the paths $\gamma _j$ and
$\gamma _{\ell}$ (defined for one and the same value of $a$, $\gamma _j$ is
followed by $\gamma _{\ell}$) and similarly for the loops (all with base point
$a$) $\gamma _j\eta _+\delta _s\eta _-$,
$\eta _+\delta _s\eta _-\gamma _j$ etc.
Thus for~$s\geq 1$, one obtains the monodromies 

$$\begin{array}{rclrcl}
  \gamma _{2s-1}&:&\xi _{4s-3}\leftrightarrow \xi _{4s-2}~,&
  \gamma _{2s}&:&\xi _{4s-1}\leftrightarrow \xi _{4s}~,\\ \\
  \delta _{2s-1}&:&\xi _{4s-2}\leftrightarrow \xi _{4s}~,&
  \delta _{2s}&:&\xi _{4s-1}\leftrightarrow \xi _{4s+1}~,\\ \\
  \gamma _{2s-1}\eta _+\delta _{2s-1}\eta _-&:&\xi _{4s-3}\mapsto \xi _{4s}~,&
  \eta _+\delta _{2s-1}\eta _-\gamma _{2s-1}&:&
  \xi _{4s}\mapsto \xi _{4s-3}~,\\ \\ 
  \gamma _{2s-1}\eta _+\delta _{2s-1}\eta _-\gamma _{2s}&:&
\xi _{4s-3}\mapsto \xi _{4s-1}~,&
\gamma _{2s-1}\eta _+\delta _{2s-1}\eta _-\gamma _{2s}\eta _+\delta_{2s}\eta _-&:&
\xi _{4s-3}\mapsto \xi _{4s+1}~{\rm etc.}
\end{array}$$
This means that, for suitably chosen loops, the root $\xi _{4s-3}$
can be mapped  
by the corresponding monodromies into any of the roots $\xi _{4s-2}$,
$\xi _{4s-1}$,
$\xi _{4s}$ or $\xi _{4s+1}$. 
After this one can repeat the reasoning with
$\xi _{4s+1}=\xi _{4(s+1)-3}$ (i.e. one can shift the value of $s$ by $1$)
and so on.

Thus the subset $S^0$ of $S$ on which all zeros of $\theta$ are simple is
connected. The set $S\setminus S^0$ belongs to the topological closure of $S$
(because the zeros of $\theta$ depend continuously on $q$), so $S$
is connected. The theorem is proved.

  \section{Proof of Theorem~\protect\ref{tm01}\protect\label{sectm01}}

  \begin{proof}[Part (1)] In the proof of parts (1) and (2) of the theorem,
    when considering the
    values of $q$ from an interval of the form $(\tilde{q}_k,\tilde{q}_{k+1})$,
    we take into account the first of formulae (\ref{eqasympt1}), so as $q$
    tends to $1^-$ (hence $k$ tends to $\infty$) one has $1-q=O(1/k)$.
    We prove first the following lemma:

  \begin{lm}\label{lmrKr}
    (1) For every $r\in (0,1)$, there exists
    $K_r\in \mathbb{N}$ such that for every $q\in (0,r]$, one has $Z_{[-e^{\pi},0)}(q)\leq K_r$.

        (2) When the zeros $\xi _{2s-1}$ and $\xi _{2s}$ are real
        (see part (2) of Remarks~\ref{rems1}), they belong to the
    interval $(-q^{-2s},-q^{-2s+1})$. 

        (3) For $q\in [\tilde{q}_k, \tilde{q}_{k+1})$, one has $Z_{[-e^{\pi},0)}(q)=o(k)$. 
        \end{lm}

  \begin{proof}
    Part (1). It
    follows from part (4) of Remarks~\ref{rems1} that
    for $q>0$ small enough, all zeros $\xi _j$ of $\theta (q,.)$ are real
    and the zeros $\xi _{2s-1}$ and $\xi _{2s}$ belong to the
    interval $(-q^{-2s},-q^{-2s+1})$, so they are smaller than~$-r^{-2s+1}$. And
    in the same way, for
    any $q\in (0,1)$, the zeros $\xi _{2s-1}$ and $\xi _{2s}$, when they are real, belong to the
    interval $(-q^{-2s},-q^{-2s+1})$ (which proves part (2)).

    When $q$ increases and becomes equal to $\tilde{q}_s$, the zeros
    $\xi _{2s-1}$ and $\xi _{2s}$ coalesce. For $q>\tilde{q}_s$,
    they form a complex conjugate pair, see part (2) of Remarks~\ref{rems1}.
    For $q\in (0,r]$ and $2s-1>\pi /\ln (1/r)$, i.e. $q^{-2s+1}\geq r^{-2s+1}>e^{\pi}$
      hence $-q^{-2s+1}<-e^{\pi}$, the zero
    $\xi _j$, $j\geq 2s-1$, is either smaller than $-e^{\pi}$ or
    it has given birth (together with $\xi _{j-1}$ or $\xi _{j+1}$ depending on
    the parity of $j$) to a complex conjugate pair. Therefore $Z_{[-e^{\pi},0)}(q)\leq [\pi /\ln (1/r)]+1$ and one can
    set $K_r:=[\pi /\ln (1/r)]+1$.

    Part (3). Suppose first that $q=\tilde{q}_k$.
    The interval $I:=[-e^{\pi},\tilde{y}_k]$ contains all real zeros
    of $\theta (\tilde{q}_k,.)$ belonging to the interval $J:=[-e^{\pi}, 0)$. The
    rightmost of these zeros which is in $I$ is the double zero $\tilde{y}_k$
    which is the result of the confluence of $\xi _{2k-1}$ and $\xi _{2k}$,
    see parts (2) and (3) of Remarks~\ref{rems1}.
    Denote by $s_0$ the smallest
    of the numbers $s$ for which $-(\tilde{q}_k)^{-2s+1}<-e^{\pi}$. Hence there
    are not more than 
    
    $$t_0:=2(s_0-1)-2(k-1)+1=2(s_0-k)+1$$
    real zeros of
    $\theta (\tilde{q}_k,.)$ in $I$ (counted with multiplicity),
    see the proof of part (1) of the present lemma. 
    One has $\tilde{q}_k=1-\pi /2k+o(1/k)$, see (\ref{eqasympt1}).
    Therefore

    $$-(\tilde{q}_k)^{-2s_0+1}<-e^{\pi}~\Leftrightarrow ~
    (-2s_0+1)\ln (\tilde{q}_k)>\pi ~\Leftrightarrow ~
    2s_0-1>\pi /(\ln (1/\tilde{q}_k))=2k+o(k)~.$$
    On the other hand it follows from the definition of $s_0$ that $2s_0-3\leq \pi /(\ln (1/\tilde{q}_k))=2k+o(k)$. Thus $s_0=k+o(k)$ and $t_0=o(k)$. Suppose now that
    $q\in (\tilde{q}_k, \tilde{q}_{k+1})$.
    Hence when one counts the real zeros of $\theta (q,.)$ in $J$,
      one should take into account that:
      \vspace{1mm}
      
      1) The double root $\tilde{y}_k$ gives birth to a complex conjugate pair
      of zeros, i.e. two real zeros are lost; for
      $q\in (\tilde{q}_k, \tilde{q}_{k+1})$, these are the only real zeros that
      are lost, see part (2) of Remarks~\ref{rems1};
      \vspace{1mm}
      
      2) Denote by $s_*(q)$ the smallest of the numbers $s$ for which one has 
      $-q^{-2s+1}<-e^{\pi}$ (hence $s_*(\tilde{q}_k)=s_0$). For fixed $s$, the
      number $-q^{-2s+1}$ increases
      with $q$, so $s_*(q)$ also increases, i.e. new real zeros might enter
      the interval $J$ from the left.
      \vspace{1mm}

      Thus for $q\in (\tilde{q}_k, \tilde{q}_{k+1})$, one has $Z_J(q)\leq t_1+2$, where $t_1$ is the quantity $t_0$
      defined for $k+1$
      instead of $k$, hence $Z_J(q)=o(k)$. Indeed, the numbers $-q^{-2s+1}$
      increase with $q$. We cannot claim that if for $s=s_*(q)-1$, one has $-q^{-2s+1}\geq -e^{\pi}$, then
      the zeros $\xi _{2s-1}$ and
      $\xi _{2s}$ are larger or smaller than $-e^{\pi}$; this is why $2$ is
      added to~$t_1$.

  \end{proof}

  The proof of part (1) of Theorem~\ref{tm01} results from part~(3) of Lemma~\ref{lmrKr}.
  Indeed, one has $k=O(1/(1-\tilde{q}_k))$, see (\ref{eqasympt1}). 

  \end{proof}

  \begin{proof}[Part (2)]The function $\theta$ satisfies the following
    functional equation:

    \begin{equation}\label{eqfunct}
      \theta (q,x)=1+qx\theta (q,qx)~.
      \end{equation}
For $q=\tilde{q}_k\in \Gamma$, we denote by $\cdots <x_2<x_1<x_0<0$ the numbers
$x_0=\tilde{y}_k$, $x_s=x_{s-1}/\tilde{q}_k$, $s\in \mathbb{N}$ (i.e.
$x_s=\tilde{y}_k/(\tilde{q}_k)^s$). Hence 
$\theta (\tilde{q}_k,x_0)=0$,
$\theta (\tilde{q}_k,x_1)=1+x_0\theta (\tilde{q}_k,x_0)=1$
(see (\ref{eqfunct})), and for $s>1$,
\vspace{1mm}

(i) if $\theta (\tilde{q}_k,x_s)<0$, then
$\theta (\tilde{q}_k,x_{s+1})=1+x_s\theta (\tilde{q}_k,x_s)>1$;
\vspace{1mm}

(ii) if $\theta (\tilde{q}_k,x_s)\geq 1$ (this is the case for $s=1$), 
then for $k$
sufficiently large, one has $x_s<-e^{\pi}/2$ (see (\ref{eqasympt1})), 
$\tilde{q}_k\in (0.3,1)$ (see parts (1) and (2) of Remarks~\ref{rems1}) hence 
$\tilde{q}_kx_s<-0.3\times e^{\pi}/2<-3$ and

$$\theta (\tilde{q}_k,x_{s+1})=1+\tilde{q}_kx_s\theta (\tilde{q}_k,x_s)<
1-3=-2<0~.$$

Thus for $k$ sufficiently large, we have $\theta (\tilde{q}_k,x_s)<0$ for
$s\geq 2$ even and $\theta (\tilde{q}_k,x_s)>0$
for $s\geq 3$ odd. Hence each interval $(x_{s+1},x_s)$ contains a zero of 
$\theta$. For a fixed interval $[-a,-e^{\pi}]$, consider the intervals $(x_{s+1},x_s)$ which are its subintervals. 
As $k\rightarrow \infty$ (hence $\tilde{q}_k\rightarrow 1^-$) the
lengths of these intervals
tend uniformly to $0$. Indeed, 
the largest of them is the last one and its length is 
$\leq (a-aq)=(1-q)a$.
Therefore for any $a>e^{\pi}$, the set of zeros of 
$\theta (\tilde{q}_k,.)$
(over all $k$ sufficiently large) is everywhere dense in the interval 
$[-a,-e^{\pi}]$.
This proves the first claim of part (2) of the theorem.
To prove the second one we first consider the case $q=\tilde{q}_k\in \Gamma$.
We define the quantities $u_0$, $u_1\in \mathbb{N}$ by the conditions 

$$\begin{array}{cccccccccll}
  |\tilde{y}_k|/q^{u_0}&=&|x_{u_0}|&\leq &e^{\pi}&<&|x_{u_0+1}|&=&
  |\tilde{y}_k|/q^{u_0+1}&&
{\rm and}\\ \\ 
|\tilde{y}_k|/q^{u_1}&=&|x_{u_1}|&\leq &a&<&|x_{u_1+1}|&=&
|\tilde{y}_k|/q^{u_1+1}&&.\end{array}$$ 
Hence (remember that $\ln q<0$) 

$$(u_0+1)\ln q<\ln (|\tilde{y}_k|/e^{\pi})\leq u_0\ln q~~~\, \, {\rm and}~~~\, \, 
(u_1+1)\ln q<\ln (|\tilde{y}_k|/a)\leq u_1\ln q$$
which, taking into account that as $q\rightarrow 1^-$, one has 
$\ln q=\ln (1+(q-1))=(q-1)+o(q-1)$, implies 

$$u_0(q-1)=\ln (|\tilde{y}_k|/e^{\pi})+o(q-1)~~~\, \, {\rm and}~~~\, \, 
u_1(q-1)=\ln (|\tilde{y}_k|/a)+o(q-1)~.$$
It is clear that 
$\ell _a(q)=u_1-u_0+O(1)$. Thus 

$$\ell _a(q)(1-q)=(u_1-u_0)(1-q)+O(1)(1-q)=\ln (a/e^{\pi})+O(1-q)~.$$
Now suppose that $q\in (\tilde{q}_k,\tilde{q}_{k+1})$. Our reasoning is
similar to the one in the proof of Lemma~\ref{lmrKr}. The double zero
$\tilde{y}_k$ gives birth to a complex conjugate pair, so two real zeros
are lost. If for $q=q_*\in (\tilde{q}_k,\tilde{q}_{k+1})$, the interval
$(-q_*^{-2s},-q_*^{-2s+1})$ is a subset of the interval $[-a,0)$,
  then the same is true
  for $q=\tilde{q}_{k+1}$. Thus
  \begin{equation}\label{eqstar}\ell _a(q_*)\leq Z_{[-a,0)}(\tilde{q}_{k+1})+2~.\end{equation} 
  One adds $2$ in order to take into account the
      two zeros of $\theta (\tilde{q}_{k+1},.)$ of the not more than one interval
      $(-q_*^{-2s},-q_*^{-2s+1})$ which belongs partially,
      but not completely, to $[-a,0)$. The number $2$ of lost zeros and the
        number $2$ in (\ref{eqstar}) are $o(1/(1-q_*))$.
        According to part (1) of the theorem
        $$Z_{[-a,0)}(\tilde{q}_{k+1})=\ell _a(\tilde{q}_{k+1})+o(1/(1-q_*))~,$$
          and for $q=\tilde{q}_{k+1}$, it was shown that
          $(1-\tilde{q}_{k+1})\ell _a(\tilde{q}_{k+1})=\ln (a/e^{\pi})+o(1)$, so
          $\ell _a(q_*)=\ell _a(\tilde{q}_{k+1})+o(1/(1-q_*))$ and
          $(1-q_*)\ell _a(q_*)=\ln (a/e^{\pi})+o(1)$ which proves part (2) of the
          theorem.

    \end{proof}

  \begin{proof}[Part (3)] We need the following lemma:
    \begin{lm}\label{lmrho} Suppose that $q\in (-1,0)$ and set $\rho :=|q|$. Then:
      
      (1) For $\rho >0$ small enough, one has
      
      \begin{equation}\label{eq4eq}
       \begin{array}{lcl}
\xi _{4s}\in (-\rho ^{-4s-1}, -\rho ^{-4s+1})~,&&
      \xi _{4s+2}\in (-\rho ^{-4s-3}, -\rho ^{-4s-1})~,\\ \\ 
      \xi _{4s-1}\in (\rho ^{-4s+2}, \rho ^{-4s})&~~~\, {\rm and}~~~\, &
      \xi _{4s+1}\in (\rho ^{-4s}, \rho ^{-4s-2})~.\end{array}\end{equation} 
      Moreover, the mentioned zeros $\xi _j$ 
      are the only zeros of $\theta (q,.)$ in the indicated intervals.

      (2) For $\rho \in (0,1)$, one has $\theta (q, -q^{-2k})=
      \theta (q, -\rho ^{-2k})\in (0, \rho ^{2k}+\rho ^{4k+1})$.

      (3) For $q\in [\bar{q}_{2s-1},0)$, the zeros $\xi _{4s-2}$ and
  $\xi _{4s}$ belong to
  the interval $I^{\bullet}:=(-\rho ^{-4s},-\rho ^{-4s+2})$.
    \end{lm}

    \begin{proof}[Proof of Lemma~\ref{lmrho}] Part (1). 
      We consider the following four series:

      $$\begin{array}{lcll}
        \theta ^{\diamond}:=\theta (-\rho ,-\rho ^{-4s+1})=\sum _{j=0}^{\infty}d_j~,&
        ~~~\, &d_j:=(-1)^{j(j+3)/2}\rho ^{-(4s-1)j+j(j+1)/2}~,&\\ \\
        \theta ^{\nabla}:=\theta (-\rho ,-\rho ^{-4s-1})=\sum _{j=0}^{\infty}h_j~,&&
  h_j:=(-1)^{j(j+3)/2}\rho ^{-(4s+1)j+j(j+1)/2}~,&\\ \\ 
        \theta ^{\heartsuit}:=\theta (-\rho ,\rho ^{-4s})=\sum _{j=0}^{\infty}r_j~,&&
        r_j:=(-1)^{j(j+1)/2}\rho ^{-4sj+j(j+1)/2}&{\rm and}\\ \\
        \theta ^{\star}:=\theta (-\rho ,\rho ^{-4s+2})=
  \sum _{j=0}^{\infty}\lambda _j~,&& 
  \lambda _j:=(-1)^{j(j+1)/2}\rho ^{-(4s-2)j+j(j+1)/2}~.\end{array}$$
      For the first series, its terms of largest modulus are $d_{4s-1}$ and
      $d_{4s-2}$; one has
  $d_{4s-1}=d_{4s-2}=-\rho ^{-8s^2+6s-1}$. The moduli of the terms decrease rapidly
      as $j>4s-1$ increases or as $j<4s-2$ decreases. In this series the sign
      $(-1)^{j(j+3)/2}$
  is positive for $j=4\nu$ and $j=4\nu +1$ and negative for $j=4\nu +2$
  and $j=4\nu +3$. Hence for $\rho$ small enough, the sign of
  $\theta ^{\diamond}$ is the same as the one of $d_{4s-1}+d_{4s-2}$, i.e. one has
  $\theta ^{\diamond}<0$.

  For the other three series the largest modulus terms are respectively
  $h_{4s}=h_{4s+1}=\rho ^{-8s^2-2s}>0$, $r_{4s-1}=r_{4s}=\rho ^{-8s^2+2s}>0$ and
  $\lambda _{4s-3}=\lambda _{4s-2}=-\rho ^{-8s^2+10s-3}<0$, so in the same way 
  $\theta ^{\nabla}>0$, $\theta ^{\heartsuit}>0$ and $\theta ^{\star}<0$. Hence
  there is at least one zero of $\theta$ in the interval
  $(-\rho ^{-4s-1}, -\rho ^{-4s+1})$. In fact, there is exactly one zero,
  and this is $\xi _{4s}$. Indeed, for $\rho$ small enough this is true, because one has
  $\xi _m\sim -q^{-m}$, see~\cite{Ko14} (the zeros $\xi _{4s-1}$ and $\xi _{4s+1}$ are positive, so only $\xi _{4s}$ belongs to $(-\rho ^{-4s-1}, -\rho ^{-4s+1})$). For any $\rho \in (0,1)$, this follows from the fact that as $\rho$ increases, new complex conjugate pairs are born, but the inverse does not take place, see part~(2) of Remarks~\ref{rems1}. In the same way one proves the rest of
  part (1) of the lemma.

  Part (2). One checks directly that

  $$\begin{array}{rcl}
    \theta (q,-\rho ^{-2k})&=&\sum _{j=0}^{\infty}q^{j(j+1)/2}(-\rho ^{-2k})^j=
  \sum _{j=0}^{\infty}(-1)^{j(j+3)/2}\rho ^{-2kj+j(j+1)/2}\\ \\ &=&
  \sum _{j=4k}^{\infty}(-1)^{j(j+3)/2}\rho ^{-2kj+j(j+1)/2}~.\end{array}$$
  The last of these equalities follows from the fact that the first $4k$ terms
  of the series cancel (the first with the $(4k)$th, the second with the
  $(4k-1)$st etc.). The signs of the terms of the last of these series are
  $+,+,-,-,+,+,-,-,\cdots$ and the exponents $-2kj+j(j+1)/2$ are increasing
  for $j\geq 4k$. Hence the series is the sum of two Leibniz series
  with positive first terms, so its sum is positive and not larger than the sum
  of the first terms of these two series. The latter sum is
  $\rho ^{2k}+\rho ^{4k+1}$ which proves part~(2).

  Part (3). For $\rho$ sufficiently small, the zeros $\xi _{4s-2}$ and
  $\xi _{4s}$ belong to $I^{\bullet}$. Indeed, by part (2) of the
  present lemma, at the 
  endpoints of $I^{\bullet}$ the function $\theta (q,.)$
  is positive while it is negative at $-\rho ^{-4s+1}$ (we showed already that
  $\theta ^{\diamond}<0$). As $\theta (q,.)$ is positive
  at the endpoints for any $q\in (-1,0)$, the zeros $\xi _{4s-2}$ and
  $\xi _{4s}$ belong to $I^{\bullet}$ exactly for $q\in [\bar{q}_{2s-1},0)$,
    see part (3) of Remarks~\ref{rems1bis}. This proves part~(3) of
    Lemma~\ref{lmrho}. 
      \end{proof}

    Suppose first that $q=\bar{q}_{2\nu -1}$, $\nu \in \mathbb{N}$.
    The rightmost of
    the negative zeros of $\theta (\bar{q}_{2\nu -1},.)$ is the double zero
    $\bar{y}_{2\nu -1}=\xi _{4\nu -2}=\xi _{4\nu}$, see part (3) of
    Remarks~\ref{rems1bis}. Denote
    by $s^{\dagger}=s^{\dagger}(\bar{q}_{2s-1})$ the largest of the
    numbers $s\in \mathbb{N}$ for which one has
    $-(\bar{q}_{2\nu -1})^{-4s}\geq -e^{\pi /2}$. Hence the zero $\xi _{4s^{\dagger}}$ is
    in the interval $[-e^{\pi /2},0)$ and the zero $\xi _{4(s^{\dagger}+1)}$
      is to its left, i. e. outside it. Thus the number
      $\tilde{N}(\bar{q}_{2\nu -1}):=Z_{[-e^{\pi /2},0)}(\bar{q}_{2\nu -1})$ (the zeros in $[-e^{\pi /2},0)$ have only even indices $i$, see
      Remarks~\ref{rems1bis}) is

      $$\tilde{N}(\bar{q}_{2\nu -1})=(4s^{\dagger}-4\nu +2)/2+u=
      2(s^{\dagger}-\nu )+1+u~,$$
      where $u\leq 1$ (the presence of the number $u$ reflects the fact that we do not say whether the zero $\xi _{4s^{\dagger}+2}$ belongs or not to the interval $[-e^{\pi /2},0)$). The conditions

      $$-(\bar{q}_{2\nu -1})^{-4(s^{\dagger}+1)}<-e^{\pi /2}\leq
      -(\bar{q}_{2\nu -1})^{-4(s^{\dagger})}$$
      are equivalent to
      $-4(s^{\dagger}+1)\ln |\bar{q}_{2\nu -1}|>\pi /2
      \geq -4s^{\dagger}\ln |\bar{q}_{2\nu -1}|$ or to

      $$\left\{ \begin{array}
        {ccl}4(s^{\dagger}+1)&>&(\pi /2)/(\ln (1/|\bar{q}_{2\nu -1}|))=
        (\pi /2)/(\ln (1+\pi /(8(2\nu -1))+o(1/\nu )))=4\nu +O(1)~,\\ \\
        4s^{\dagger}&\leq&(\pi /2)/(\ln (1/|\bar{q}_{2\nu -1}|))~,
        \end{array}\right.$$
      see the first of formulae (\ref{eqasympt2}). Thus

      \begin{equation}\label{eqsN}
        s^{\dagger}=\nu +O(1)~~~\, 
        {\rm and}~~~\, \tilde{N}(\bar{q}_{2\nu -1})=O(1)~.
      \end{equation}
      One can also write
      $\tilde{N}(\bar{q}_{2\nu -1})=o(\nu )=o(1/(1+\bar{q}_{2\nu -1}))$.
      Hence $\tilde{N}(\bar{q}_{2\nu +1})=o(\nu )$.

      Now suppose that
      $q\in (\bar{q}_{2\nu +1},\bar{q}_{2\nu -1})$. When counting the zeros
      $\xi _i$ in the interval $[-e^{\pi /2},0)$ one takes into account that the
        double zero $\xi _{4\nu -2}=\xi _{4\nu}$ is lost
        (it gives birth to a complex conjugate pair). The numbers $-\rho ^{-4s}$
        (which are left endpoints of intervals $I^{\bullet}$) increase, so new
        zeros $\xi _i$ might enter the interval $[-e^{\pi /2},0)$ from the left.
          The number of such intervals $I^{\bullet}$ which belong entirely to
          $[-e^{\pi /2},0)$ is not greater than their number for
            $q=\bar{q}_{2\nu +1}$. There is at most one interval $I^{\bullet}$
            which belongs only partially to $[-e^{\pi /2},0)$, so ignoring it
              means not counting at most $2$ zeros $\xi _i\in [-e^{\pi /2},0)$.
                Therefore
                $\tilde{N}(q)=\tilde{N}(\bar{q}_{2\nu +1})+O(1)=
                o(\nu )=o(1/(1+q))$. Part~(3) of Theorem~\ref{tm01} is proved.
                
  \end{proof}

  \begin{proof}[Part (4)]
    Consider an interval of the form $[-a,-e^{\pi /2}]$ and its subinterval
    $(-a^*,-a^{\triangle})$, $e^{\pi /2}<a^{\triangle}<a^*<a$. For $\nu \in \mathbb{N}$
    sufficiently large, the double zero
    $\bar{y}_{2\nu -1}=\xi _{4\nu -2}=\xi _{4\nu}$ of 
    $\theta (\bar{q}_{2\nu -1},.)$ is to the right of $-a^{\triangle}$ (see the second of formulae~(\ref{eqasympt2})) and there
    exists an interval of the form $I^{\bullet}$ (see Lemma~\ref{lmrho})
    such that 
    $I^{\bullet}\subset
    (-a^*,-a^{\triangle})$. Indeed,
    the length of $I^{\bullet}$
    equals $\rho ^{-4s}(1-\rho ^2)$. For each $s$ sufficiently large, one can
    choose $\rho \in (0,1)$ such that
    \begin{equation}\label{eqstarstar}
-\rho ^{-4s}\in (-a^*,(-a^*-a^{\triangle})/2)~.\end{equation}
If one chooses a larger $s$, then 
    one can achieve condition (\ref{eqstarstar}) by choosing $\rho$ closer to $1$. This
    means that, as $\rho ^{-4s}$ remains bounded, the length of $I^{\bullet}$
    tends to $0$ and one can attain
    both conditions (\ref{eqstarstar})  and $-\rho ^{-4s+2}\in (-a^*,-a^{\triangle})$. Thus
    $\xi _{4s-2}$, $\xi _{4s}\in (-a^*,-a^{\triangle})$, see part~(3) of
    Lemma~\ref{lmrho}. This proves the first claim
    of part~(4) of Theorem~\ref{tm01}.

    To prove the second claim, for $q^*\in (-1,0)$,
    we denote by $s^{\sharp}(q^*)$
    the value of $s\in \mathbb{N}$ corresponding to the leftmost of the
    numbers $-(q^*)^{-4s}$ belonging to the interval $[-a,0)$.
      In the proof of part~(4) of Theorem~\ref{tm01} we set $\rho :=|q^*|$,
      so $-(q^*)^{-4s}=-\rho ^{-4s}$.
      Hence

      $$\lim _{\rho \rightarrow 1^-}(-\rho ^{-4s^{\sharp}(q^*)})=-a~~~,~~~\, \,
      -\rho ^{-4s^{\sharp}(q^*)}>-a~~~\, \, {\rm and}~~~\, \,
      -\rho ^{-4(s^{\sharp}(q^*)+1)}<-a~.$$
      From the latter two inequalities, 
      having in mind that $\ln (1/\rho )=(1-\rho )+o(1-\rho )$, one gets

      \begin{equation}\label{eqsim}
        s^{\sharp}(q^*)\sim (\ln a)/(4(1-\rho ))~.
        \end{equation}

      Now we partition the zeros of $\theta (q^*,.)$ with negative real parts
      in several sets (we remind that there are no zeros of
      $\theta (q^*,.)$ on the imaginary axis for any $q^*\in (-1,0)$,
      see~\cite{Ko15}):
      \vspace{1mm}

      1) The set $S_{\infty}$ of zeros $\xi _j$ belonging to the intervals
      $I^{\bullet}$ with $s\geq s^{\sharp}(q^*)+2$. These zeros (when considered
      as depending continuously on $q\in [q^*,0)$) are real and do not
        belong to the interval $[-a,0)$ for any $q\in [q^*,0)$.
            \vspace{1mm}

          2) The set $S_0$ of the two zeros of the interval $I^{\bullet}$ with
          $s=s^{\sharp}(q^*)+1$.
          \vspace{1mm}

          3) The set $S_R$ of the other real negative zeros of $\theta (q^*,.)$. We
          subdivide this set into $S_R([-a,-e^{\pi /2}])$ and
          $S_R((-e^{\pi /2},0))$ of zeros belonging to the respective intervals.
          \vspace{1mm}

          4) The set $S_I$ of the complex conjugate pairs of zeros of
          $\theta (q^*,.)$ which have negative real parts. For $q^*\in (\bar{q}_{2\nu +1},\bar{q}_{2\nu -1})$,
          their number is~$\nu$. For $q^*<0$ close to zero, the zeros of the set
          $S_I$ are real and belong to intervals $I^{\bullet}$, and as $q^*$
          decreases, they form complex conjugate pairs, see
          Remarks~\ref{rems1bis}.  
          \vspace{1mm}

          By abuse of notation we denote by the same symbols sets
          (e.g. $S_I$, $S_R$ etc.) and the number of zeros of $\theta$ which
          they contain. We remind that the numbers
          $n_a(q^*)$ and $s^{\dagger}(q^*)$
          are defined in Notation~\ref{notalnp} and in the proof of part~(3)
          of the present theorem respectively; the number $s^{\dagger}(q^*)$
          satisfies the first of conditions (\ref{eqsN}).  Hence for
          $q^*\in (\bar{q}_{2\nu +1},\bar{q}_{2\nu -1})$, one has

          \begin{equation}\label{eqend1}
            n_a(q^*)=S_R([-a,-e^{\pi /2}])+A~,
            \end{equation}
          where $A=0$, $1$ or $2$ is the number of zeros of the set $S_0$ which
          belong to the interval $[-a,-e^{\pi /2}]$. On the other hand,

          \begin{equation}\label{eqend2}
            S_R([-a,-e^{\pi /2}])=2s^{\sharp}(q^*)-S_R((-e^{\pi /2},0))-S_I~.
            \end{equation}
          Recall that $S_I=2\nu$. By the first of equations (\ref{eqsN}) one has 
          $\nu =s^{\dagger}(q^*)+O(1)$, and by part~(3) of the present theorem one has 
          $S_R((-e^{\pi /2},0))=o(\nu )$. That's why equations (\ref{eqend1})
          and (\ref{eqend2}) imply

          \begin{equation}\label{eqend3}
            n_a(q^*)=2s^{\sharp}(q^*)-2s^{\dagger}(q^*)+o(\nu )~.
          \end{equation}
          The factor $2$ corresponds to the fact that there are two zeros
          $\xi _i$ in the interval $I^{\bullet}$. One can apply formula
          (\ref{eqsim}) with $a=e^{\pi /2}$ to obtain
          $s^{\dagger}(q^*)\sim \ln (e^{\pi /2})/(4(1-\rho ))$ and from
          (\ref{eqend3}) one concludes that
          $n_a(q^*)=(\ln (a/e^{\pi /2}))/(2(1-\rho ))+o(1/(1-\rho ))$
          from which part~(4) of the theorem follows.

  \end{proof}

  \begin{proof}[Parts (5) and (6)]
    We begin by proving the first claim of part~(6); in this part of the proof
    we write $q$ instead of
    $\bar{q}_{2s}$. For any $\varepsilon >0$, there exists
    $s^{\nabla}\in \mathbb{N}$ such that
    for $s\geq s^{\nabla}$, one has $\bar{y}_{2s}\in (e^{\pi /2}-\varepsilon$,
    $e^{\pi /2}+\varepsilon )$, see formulae~(\ref{eqasympt2}). We assume that $\varepsilon <1/2$, so
    $\bar{y}_{2s}>3$. For $s\geq s^{\nabla}$, we set
    $x_j:=\bar{y}_{2s}/q^j$, $j\in \mathbb{N}$. One has $\theta (q,x_0)=0$, $x_{2m}>0$, 
    $x_{2m+1}<0$ and $|x_j|>3$. Therefore

    $$\begin{array}{lcl}
      \theta (q,x_1)&=&1+x_0\theta (q,x_0)=1>0~,\\ \\
      \theta (q,x_2)&=&1+x_1\theta (q,x_1)=
      1+x_1(1+x_0\theta (q,x_0))<1-3=-2<0~,\\ \\
      \theta (q,x_3)&=&1+x_2\theta (q,x_2)<-2<0~~~\, {\rm and}\\ \\
      \theta (q,x_4)&=&1+x_3\theta (q,x_3)>2>0~.
    \end{array}$$
    In the same way one shows that $\theta (q,x_{4m+2})<-2<0$ and
    $\theta (q,x_{4m+4})>2>0$. Hence at least one zero of $\theta (q,.)$
    belongs to the interval $(x_{4m+2},x_{4m+4})$. The longest of these intervals
    for which $x_{4m+2}\in [e^{\pi /2}+\varepsilon ,a]$ is the last one,
    i.e. the one with
    largest value of $m$. Its length is $\leq a((1/q^2)-1)$ which quantity
    tends to $0$ as $q\rightarrow -1^+$ (i.e. as $s\rightarrow \infty$).
    Hence the zeros of $\theta (q,.)$ are everywhere dense in the interval
    $[e^{\pi /2}+\varepsilon ,a]$, and as $\varepsilon >0$ is arbitrary, they are
    everywhere dense in $[e^{\pi /2},a]$. This proves the first claim of part~(6).

    To prove part (5) we observe that for $x\in (\xi _{4s+4},\xi _{4s+2})$, one has
    $\theta (q,x)\leq 0$ and according to (\ref{eqfunct}),
    $\theta (q,qx)=\theta (q,x)/(qx)-1/(qx)<0$ (because $qx>0$). Hence 
    $(q\xi _{4s+2},q\xi _{4s+4})\subset (\xi _{4s+1},\xi _{4s+3})$,
    see Fig.~3 in~\cite{Ko7} (in~\cite{Ko7} the latter inclusion is proved
    only for $q\in [-0.108, 0)$; for any $q\in (-1,0)$, provided that the
      zeros $\xi _{4s+1}$, $\xi _{4s+2}$, $\xi _{4s+3}$ and $\xi _{4s+4}$ are real,
      it follows by continuity). Thus

    $$Z_{(0,e^{\pi /2}]}(q)=Z_{[-e^{\pi /2}/|q|,0)}(q)+B=Z_{[-e^{\pi /2},0)}(q)+
        Z_{[-e^{\pi /2}/|q|,-e^{\pi /2})}(q)+B~,$$
          where $B=-1$, $0$ or $1$ indicates that the count might not concern
          the leftmost zero in $[-e^{\pi /2}/|q|,0)$ and/or the rightmost zero in
            $(0,e^{\pi /2}]$. By parts (3) and (4) of the present theorem
          each of the summands $Z_{[-e^{\pi /2},0)}(q)$ and $Z_{[-e^{\pi /2}/|q|,-e^{\pi /2})}(q)$
              is $o(1/(1+q))$ which proves part~(5). In the
                same way one proves the second claim of part~(6) as well:

                $$Z_{[e^{\pi /2},a]}(q)=Z_{[-a/|q|,-e^{\pi /2}/|q|]}(q)+B=
                  Z_{[-a/|q|,-e^{\pi /2}]}(q)-Z_{(-e^{\pi /2}/|q|,-e^{\pi /2}]}(q)+B~,$$
                      where $Z_{(-e^{\pi /2}/|q|,-e^{\pi /2}]}(q)=o(1/(1+q))$ and
            $$Z_{[-a/|q|,-e^{\pi /2}]}(q)=(\ln ((a/|q|)/e^{\pi /2})/2)/(1+q)=
                          (\ln (a/e^{\pi /2})/2)/(1+q)+o(1/(1+q))~.$$
                          The theorem is proved.
    
    \end{proof}

\end{document}